\documentclass[12pt,reqno]{amsart}

\usepackage{amssymb,array}
\usepackage{amsmath}
\usepackage{amsthm}
\usepackage{amsbsy,mathrsfs}
\usepackage{bm}
\usepackage[english]{babel}
\usepackage{graphicx}
\usepackage{color}
\usepackage{hyperref}
\usepackage[toc,page]{appendix}
\usepackage{thmtools}
\usepackage{thm-restate}
\usepackage{bbm}
\usepackage[shortlabels]{enumitem}
\usepackage{cleveref}
\usepackage{xcolor}

\divide
\oddsidemargin by 2 
\theoremstyle{plain}
\newtheorem{theorem}{Theorem}[section]
\newtheorem*{theorem*}{Theorem}

\newtheorem{lemma}[theorem]{Lemma}

\newtheorem{proposition}[theorem]{Proposition}

\newtheorem{computational theorem}[theorem]{Computational Theorem}

\theoremstyle{definition}

\newtheorem{remark}[theorem]{Remark}

\newtheoremstyle{named}{}{}{\itshape}{}{\bfseries}{.}{.5em}{\thmnote{#3 }#1}
\theoremstyle{named}

\theoremstyle{plain} 
\newcommand{\thistheoremname}{}
\newtheorem*{genericthm*}{\thistheoremname}

\begin{document}

\title[Turaev cobracket]{Turaev cobracket, synthetic geometry and applications}

\author{Moira Chas}
\address{Stony Brook Mathematics Department and Institute for Mathematical Sciences}
\email{moira.chas@stonybrook.edu}

\author{ Arpan Kabiraj}
\address{Department of Mathematics, Indian Institute of Technology Palakkad}
\email{arpaninto@iitpkd.ac.in}
\begin{abstract}

A teardrop geodesic on a hyperbolic surface is a closed curve that is geodesic except at exactly one point, called its vertex. One imagines a closed curve with one point held down as it deforms to a geodesic. For each teardrop geodesic, we give an explicit geometric construction of the closed geodesic in its free homotopy class. This construction also yields a relation between the length of the teardrop geodesic, the angle at its vertex, and the length of its geodesic representative. We then apply the construction to a self-intersecting closed geodesic at a chosen self-intersection point. Since, at that point, the geodesic splits into two directed teardrop geodesics, our construction gives a geometric description of Turaev’s cobracket in terms of geodesic representatives. Using this description we prove that if two self-intersection points yield a pair of canceling terms in the cobracket, then the angle at these points and the lengths of the small loops determined are the same. 

\end{abstract}
\maketitle
\section{Introduction}

In 1986 Goldman discovered a (fascinating in our opinion) Lie algebra in the free module generated by free homotopy classes of curves on a surface, \cite{goldman_invariant_1986}. The Lie bracket is  defined by combining two familiar operations: the based loop product and the intersection of immersed submanifolds.  More precisely, given two directed closed curves on the surface,  $\alpha$ and $\beta$ intersecting transversely in double points, each intersection point $P$ determines the free homotopy class of the loop product $\alpha\cdot_P\beta$, with a sign determined by the orientation of the two tangent vectors at $P$. The Goldman bracket is the signed sum of these classes over all intersection points.

A few years later,  \cite{turaev1991skein} Turaev complemented this structure by defining a Lie coalgebra, defined by combining, again, two well-known although more subtle operations, the based coproduct of loops and the self-intersection of immersed submanifolds. More precisely,  consider an oriented closed curve $\alpha$ that self-intersects only in transversal double points. Given one of these self-intersection points $P$, order the two tangent vectors of $\alpha$ at $P$ so that their order agrees with the orientation of the surface. Let $\alpha_1^P$ and $\alpha_2^P$ be the loops that start at $P$ in the directions of the first and second tangent vectors, respectively, follow $\alpha$ in its given orientation, and stop upon their first return to $P$.  The \emph{Turaev cobracket} is the linear extension to the free module generated by the free homotopy classes of non-contractible loops of the map defined on a generator $\langle\alpha\rangle$ by
$$
\delta\big(\langle\alpha\rangle\big)=\sum_{P\in\cap\alpha}
\Big(\langle\alpha^P_1\rangle_0\otimes\langle\alpha^P_2\rangle_0
-\langle\alpha^P_2\rangle_0\otimes\langle\alpha^P_1\rangle_0\Big),
$$
where $\langle\gamma\rangle_0$ denotes the free homotopy class of $\gamma$ if $\gamma$ is not contractible, and $0$ otherwise.. 
Turaev proved \cite{turaev1991skein} that the cobracket is independent of the choice of  representative of the free homotopy class $\langle\alpha\rangle$,  and satisfies the identities of a Lie coalgebra. 

Moreover, Turaev \cite{turaev1991skein}  showed that two structures, the Goldman Lie bracket and the Turaev cobracket, satisfy a compatibility relation, and thus they form a Lie bialgebra.

A natural question is the following: how much of the topology can be recovered from the algebra? In other words, if one knows the Lie bialgebra, what can one say about the minimal intersection and self-intersection structure of curves on the surface?

\subsection{What the Goldman–Turaev Lie bialgebra knows about intersections}

There is a substantial body of work indicating that this algebraic structure retains  abundant topological and geometric information. Goldman showed that when one of two free homotopy classes is represented by a simple closed curve, the vanishing of their Goldman bracket detects whether the two classes admit disjoint representatives \cite{goldman_invariant_1986}. This was strengthened by Chas, who showed that in this situation there is no cancellation in the Goldman bracket: the number of terms, counted with multiplicity, is precisely the geometric intersection number \cite{chas2010minimal}. For pairs of  non-simple curves, however, cancellations may occur, and the Goldman bracket can vanish even when the geometric intersection number is positive \cite{chas2004combinatorial, chas2010minimal}. Nevertheless, this loss of  information can largely be recovered by considering suitable powers and refinements of the bracket. Chas and Krongold \cite{chas2010algebraic} obtained algebraic characterizations of simple closed curves using Goldman brackets with powers on surfaces with boundary.  Chas and Gadgil \cite{chas2016extended} showed that an extended Goldman bracket can recover mutual intersection numbers and self-intersection numbers from suitable powers of curve classes. These results imply that although intersection data may not always be visible from a single bracket computation, it is  encoded in the Goldman Lie algebra.

There are parallel results concerning the Turaev cobracket and self-intersection of curves. Turaev \cite{turaev1991skein} observed that the cobracket of a power of a simple curve is zero, and asked whether the converse holds. This question was particularly inspiring to of one of us (Chas)  who was already thinking about cancellation of terms in the Goldman bracket. Chas answered Turaev's question negatively \cite{chas2004combinatorial}  for surfaces of positive genus and negative Euler characteristic by giving a family of examples of non-simple, non-power closed curves with cobracket zero. Le Donne \cite{ledonne2008lie} showed that the answer is positive in genus zero. 
As in the bracket case, geometric information appears when considering  powers: For surfaces with boundary, Chas and Krongold \cite{chas2015algebraic} proved that the vanishing of the Turaev cobracket of a power larger than two of a curve implies this curve is a power of a simple closed curve. Moreover, the number of terms in the cobracket of sufficiently high powers determines the minimal self-intersection number of a non proper power class. In a related direction, Cahn \cite{cahn2013generalization} introduced a refinement of the Turaev cobracket which gives an exact formula for the minimal self-intersection number. Taken together, these results strongly suggest that the Goldman-Turaev Lie bialgebra is not  merely  an algebraic structure associated with a surface, but also a repository of geometric and topological information about its curves. In particular, a  description of the Turaev cobracket by synthetic geometry provides a natural framework for investigating precisely which features of the minimal self-intersection structure survive, and how they can be recovered algebraically.

Further evidence that the Goldman algebra retains subtle geometric information is provided by Kabiraj \cite[Theorem 5.1]{kabiraj2018equal}.
Let $x$ and $y$ be oriented closed geodesics on a hyperbolic surface, and let $p,q\in x\cap y$. Kabiraj proved that if the two intersection points determine the same free homotopy class in the Goldman bracket,
then the appropriately oriented angles of intersection at $p$ and $q$ are equal for every hyperbolic metric on the
surface.  In particular, this result provides a geometric obstruction to cancellation among the
terms of the Goldman bracket.

\subsection{What we are adding here}

Throughout, $\Sigma$ denotes an oriented surface equipped with a complete hyperbolic metric (constant curvature $-1$) with no cusps and  possibly empty geodesic boundary. For a closed directed curve $y$ on $\Sigma$ that is not homotopically trivial, we denote by $y^\sharp$ the unique closed geodesic in its free homotopy class. If $y$ is a piecewise geodesic, we denote its length by $\ell_y$.

We give a new, purely geometric construction of the terms of Turaev's cobracket in terms of hyperbolic geometry. This is the cobracket counterpart of the geometric definition of the Goldman bracket in terms of closed geodesics given in \cite{chas2023explicitliebracketclosed}.

\begin{figure}[ht]
	\centering
	\includegraphics[trim={20 260 20 20},  width=0.8 \textwidth]{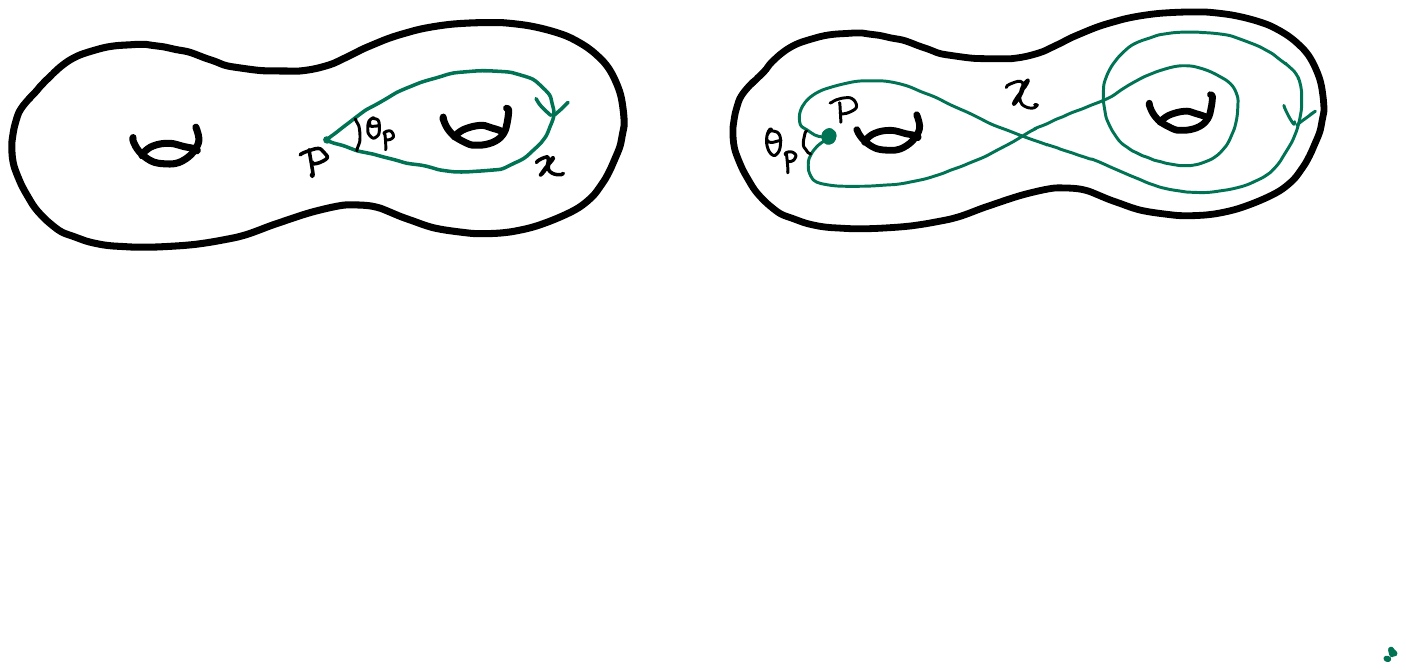}
\caption{Teardrop geodesics.}
	\label{fig:teardrop}
\end{figure}

The basic objects are teardrop geodesics. A \emph{teardrop geodesic} is a directed closed curve on $\Sigma$ that is geodesic except at exactly one point, its \emph{vertex}, (see Figure~\ref{fig:teardrop}, for a  schematic illustration) . If the vertex is $P$, the curve forms there an angle $\theta_P \in (0,\pi)$. For each teardrop geodesic $x$ with vertex $P$, we give an explicit construction, in the universal cover, of a lift of $x^\sharp$. The construction uses only reflections in geodesics determined by $x$. As a consequence we obtain (Theorem~\ref{thm:teardrop_metric}) the relation
\[
\sin\frac{\theta_P}{2}\cosh\frac{\ell_x}{2}=\cosh\frac{\ell_{x^\sharp}}{2}.
\]

Now let $x$ be a closed directed geodesic and $P$ a transverse self-intersection point of $x$. (For simplicity, we assume here that $x$ is not a proper power; the case of proper powers is addressed in Section~\ref{sec:definition}) Cutting $x$ at $P$ produces two directed teardrop geodesics $x_1^P$ and $x_2^P$, both with vertex $P$ and vertex angle $\theta_P$. They are ordered so that the branch of $x_1^P$ leaving $P$, followed by the branch of $x_2^P$ leaving $P$, agrees with the orientation of $\Sigma$ (see Figure~\ref{fig:splitting}). Applying our construction to $x_1^P$ and $x_2^P$ gives explicit lifts of $(x_1^P)^\sharp$ and $(x_2^P)^\sharp$. Define
\[
\delta'(x)=\sum_{P}\Bigl((x_1^P)^\sharp\otimes(x_2^P)^\sharp-(x_2^P)^\sharp\otimes(x_1^P)^\sharp\Bigr),
\]
where the sum runs over the transverse pairs of branches of $x$. We prove (Theorem~\ref{thm:isomorphism}) that the identification of free homotopy classes with closed geodesics is an isomorphism of Lie coalgebras between Turaev's coalgebra and the operation $\delta'$ defined on the linear combination of closed geodesics. In other words, Turaev's cobracket can be computed entirely from geodesics and their explicit hyperbolic geometry.

As an application, we obtain a rigidity result for cancellation (Proposition~\ref{prop:cancellation_rigidity}). Let $x$ be a non-simple free homotopy class in $\Sigma$ that is not a proper power. For any metric as above, the geodesic representative of $x$ realizes the minimal self-intersection number of $x$ \cite{hass1985intersections}. Moreover, its self-intersection points can be naturally identified as the metric varies.  Suppose the terms of the cobracket of $x$ corresponding to two self-intersection points $P$ and $Q$ cancel. Then, for every  hyperbolic metric on $\Sigma$ ,
\[
\theta_P=\theta_Q,\qquad
\ell_{x_1^P}=\ell_{x_2^Q},\qquad
\ell_{x_2^P}=\ell_{x_1^Q}.
\]
That is, the self-intersection angles at $P$ and $Q$ are equal, and the lengths of the matched teardrop geodesics are equal. This may be viewed as the cobracket analogue of Kabiraj's theorem \cite[Theorem 5.1]{kabiraj2018equal} on angle equality for terms of the Goldman bracket.

Since our results have necessary conditions for cancellation of a pair of terms of Turaev's cobracket, they may help in finding a characterization of classes of curves in its kernel. In relation to this, it follows from Chas' observation in \cite[Example 5.7]{chas2004combinatorial}, that on a torus with one boundary component, curves of the form $a^ib^j$, for $i, j \ge 2$, where $a$ and $b$ are two simple curves intersecting exactly once, have cobracket zero. Further computer experiments suggest that, on this surface, every non-simple class in the kernel that is not a proper power is of this type.

It is tempting to conjecture that, on a general surface, every closed geodesic with zero Turaev cobracket that is neither simple nor  a proper power is contained in an essential subsurface homeomorphic to a torus with one boundary component, and has the form described in the conjecture above. However, computational limitations have prevented us from testing this possibility , since the word lengths of relevant free homotopy classes required for informative experiments become too large on surfaces of genus greater than one.

The Goldman–Turaev Lie bialgebra was extended to all manifolds in String Topology \cite{chas1999string}. Perhaps some of the geometry above can be relevant for certain geometrical manifolds in higher dimensions.



\section{Teardrop geodesic}
\label{sec:teardrop}

Let $\Sigma$ be an orientable surface with a complete metric of constant negative curvature $-1$, no cusps and (possibly empty) geodesic boundary. Observe that, since $\Sigma$ is not necessarily compact, it can have infinite genus.

The universal cover of the surface $\Sigma$ is identified with the hyperbolic plane $\mathbb{H}$ or with a co"nvex subset of $\mathbb{H}$ if the surface has non-empty boundary.  We regard its deck transformations as restrictions of a Fuchsian group acting on \(\mathbb H\). All constructions below are carried out in \(\mathbb H\), containing the universal cover as a convex subset when the surface has non-empty boundary.

The axis of a hyperbolic transformation  $T$  on  $\mathbb{H}$ is denoted by $A_T$, and the length of a finite piecewise geodesic $d$ is denoted by $\ell_d$..  

A \emph{teardrop geodesic} $x$ is a directed closed curve on $\Sigma$ that follows a geodesic path except at exactly one point, called its \emph{vertex}, as shown in Figure \ref{fig:teardrop}.  Note that a teardrop geodesic can self-intersect and that the angle at the vertex cannot be zero. 

If $x$ is a teardrop geodesic, we denote by $x^\sharp$ the unique closed geodesic freely homotopic to $x$. If the free homotopy class of $x$ is the $n$-th power of another class, then $x^\sharp$ denotes the corresponding geodesic traversed $n$ times; in particular, its length is counted with multiplicity.

The curves drawn on surfaces (Figures~\ref{fig:teardrop}, \ref{fig:teardroproj} and \ref{fig:splitting}) throughout this paper are schematic, to help the reader visualize the objects. The figures in the hyperbolic plane (Figures~\ref{fig:teardroplift}, \ref{fig:teardropperp}, \ref{fig:teardroppent}, \ref{fig:quadrilateral} and \ref{fig:cobracketgeom}) depict  configurations in standard models of the hyperbolic plane.

Consider a teardrop geodesic $x$ with vertex $P$. The branch of $x$ arriving at $P$ and the branch of $x$ leaving $P$ determine two angles at $P$, whose sum is $2\pi$.  We denote by $\theta_P$ the smaller of these two angles; in particular $\theta_P \in (0,\pi)$.

We now describe a construction that yields a lift of the geodesic representative in the free homotopy class of $x$ in $\Sigma$ to the hyperbolic plane $\mathbb{H}$.  

Choose a lift $P'$ of the point $P$ in $\mathbb{H}$, and consider the two directed lifts $x_1$ and $x_2$ of $x$ passing through $P'$,  as shown in Figure \ref{fig:teardroplift}.  These two lifts are directed geodesic segments of length $\ell_x$,  $x_1$ directed towards $P'$ and $x_2$, leaving $P'$. The angle at $P'$, smaller than $\pi$, between $x_1$ and $x_2$
is equal to $\theta_P$. The geodesic ray that bisects $\theta_P$ is denoted by $c.$

For each $i \in \{1, 2\}$, denote by $R_i$ the midpoint of $x_i$, and by $r_i$ the perpendicular to $x_i$ passing through $R_i$. See Figure \ref{fig:teardroplift}.

\begin{figure}[ht]
	\centering
\includegraphics[trim={2cm 1cm 2cm 1cm}, width=0.4\textwidth]{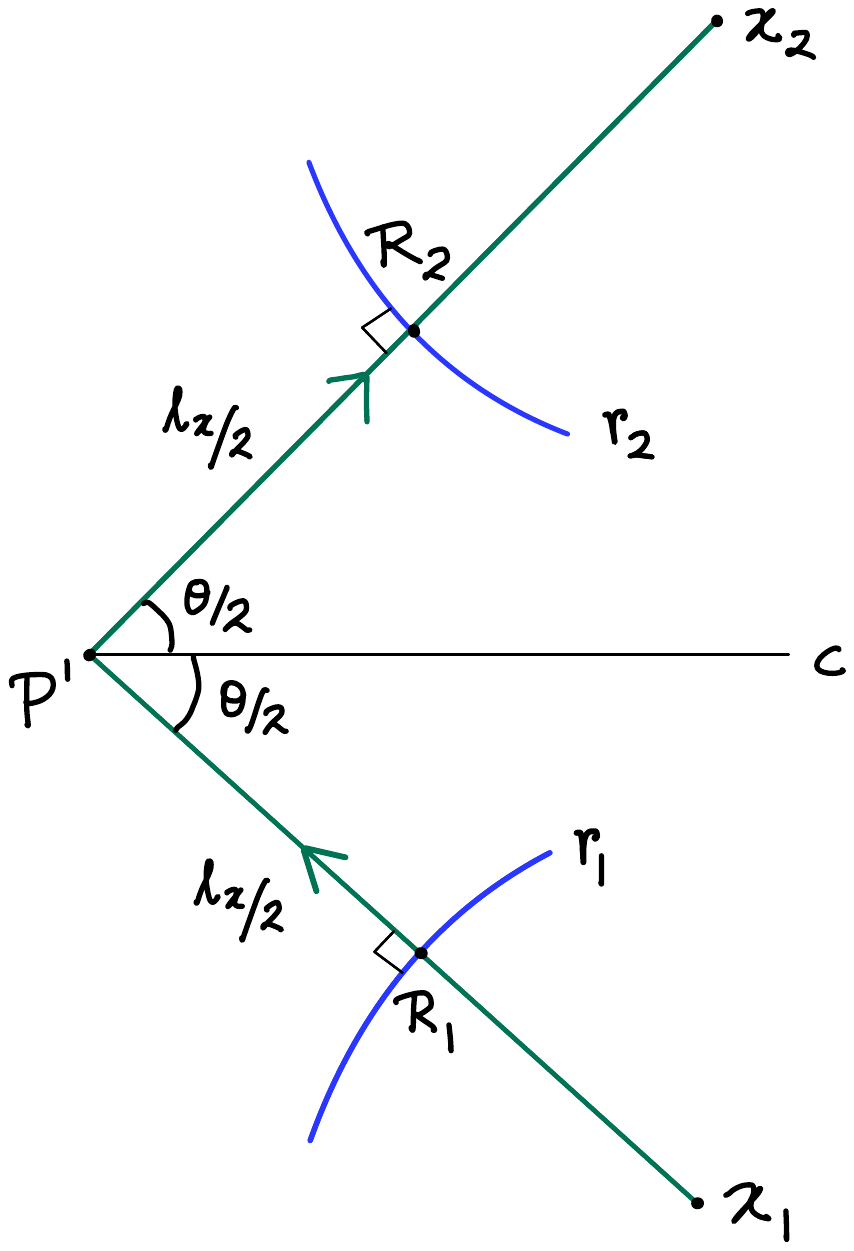}
\caption{Lifts the teardrop geodesics $x_1$ and $x_2$ of $x$ through $P'$, the perpendiculars $r_1$ and $r_2$ at their midpoints, and the bisector $c$.}
	\label{fig:teardroplift}
\end{figure}

For any geodesic segment, ray, or complete geodesic $a$ in $\mathbb H$, let $\rho_a$ denote reflection in the unique complete geodesic containing $a$.

\begin{lemma}\label{lemma:common_perp}
    The geodesics $r_1$,  $r_2$, and the ray $c$ admit a common perpendicular. 
\end{lemma}

\begin{proof}

Since the ray $c$ bisects the angle at $P'$ between the segments $x_1$ and $x_2$, and these two segments have the same length $\ell_x$, the reflection $\rho_c$ maps the segment $x_1$ onto the segment $x_2$, reversing direction, and maps the midpoint $R_1$ to the midpoint $R_2$ (see Figure~\ref{fig:teardroplift}). Consequently $\rho_c$ maps the geodesic $r_1$ to the geodesic $r_2$. On the other hand, the reflection $\rho_{r_2}$ maps the segment $x_2$ onto itself, reversing direction and fixing $R_2$. Therefore the orientation-preserving isometry
\[
U=\rho_{r_2}\circ\rho_c
\]
maps the directed segment $x_1$ onto the directed segment $x_2$, and maps $R_1$ to $R_2$.

Both $x_1$ and $x_2$ project to the directed teardrop geodesic $x$, and the points $R_1$ and $R_2$ project to the same point of $\Sigma$. Hence there exists a covering transformation mapping the directed segment $x_1$ onto the directed segment $x_2$. Since an orientation-preserving isometry of $\mathbb{H}$ is uniquely determined by the image of a directed geodesic segment, this covering transformation coincides with $U$. Thus $U=\rho_{r_2}\circ\rho_c$ is a covering transformation. Moreover, $U$ is not the identity, because the segments $x_1$ and $x_2$ are distinct (the angle $\theta_P$ lies in $(0,\pi)$).

We now examine the relative positions of the geodesic containing the ray $c$ and the geodesic $r_2$; note that these are distinct geodesics, since both intersect the geodesic containing $x_2$ at different points. There are three possibilities.

If the two geodesics intersected at a point $M$ of $\mathbb{H}$, then $U$, being the product of the reflections through two intersecting geodesics, would be an elliptic isometry fixing $M$. This is impossible, because deck group is torsion-free and hence contains no nontrivial elliptic element.

If the two geodesics were asymptotic, sharing an endpoint on the boundary of $\mathbb{H}$, then $U$, being the product of the reflections through two asymptotic geodesics, would be a parabolic isometry fixing the boundary endpoint. A parabolic covering transformation produces a cusp on $\Sigma$, contradicting our assumption that $\Sigma$ has no cusps.

Hence, the geodesic containing the ray $c$ and the geodesic $r_2$ have a common perpendicular $x'$. 
Since $x'$ is perpendicular to the geodesic containing $c$, it is invariant under $\rho_c$; hence $x'=\rho_c(x')$ is also perpendicular to $\rho_c(r_2)=r_1$. Thus $x'$ is a common perpendicular of $r_1$, $r_2$, and $c$, as desired.


\end{proof}

The segments $r_1$ and $r_2$ have a common perpendicular by Lemma~\ref{lemma:common_perp}, which we denote by $x'$. We orient $x'$ from $r_1$,  to $r_2$ as shown in Figure \ref{fig:teardropperp}.

 Define the composition $X = \rho_{r_1} \circ \rho_c$. Since reflections are orientation-reversing isometries, the composition of these two reflections, $X$, is an orientation-preserving isometry.

\begin{figure}[ht]
	\centering
	\includegraphics[width=0.5\textwidth]{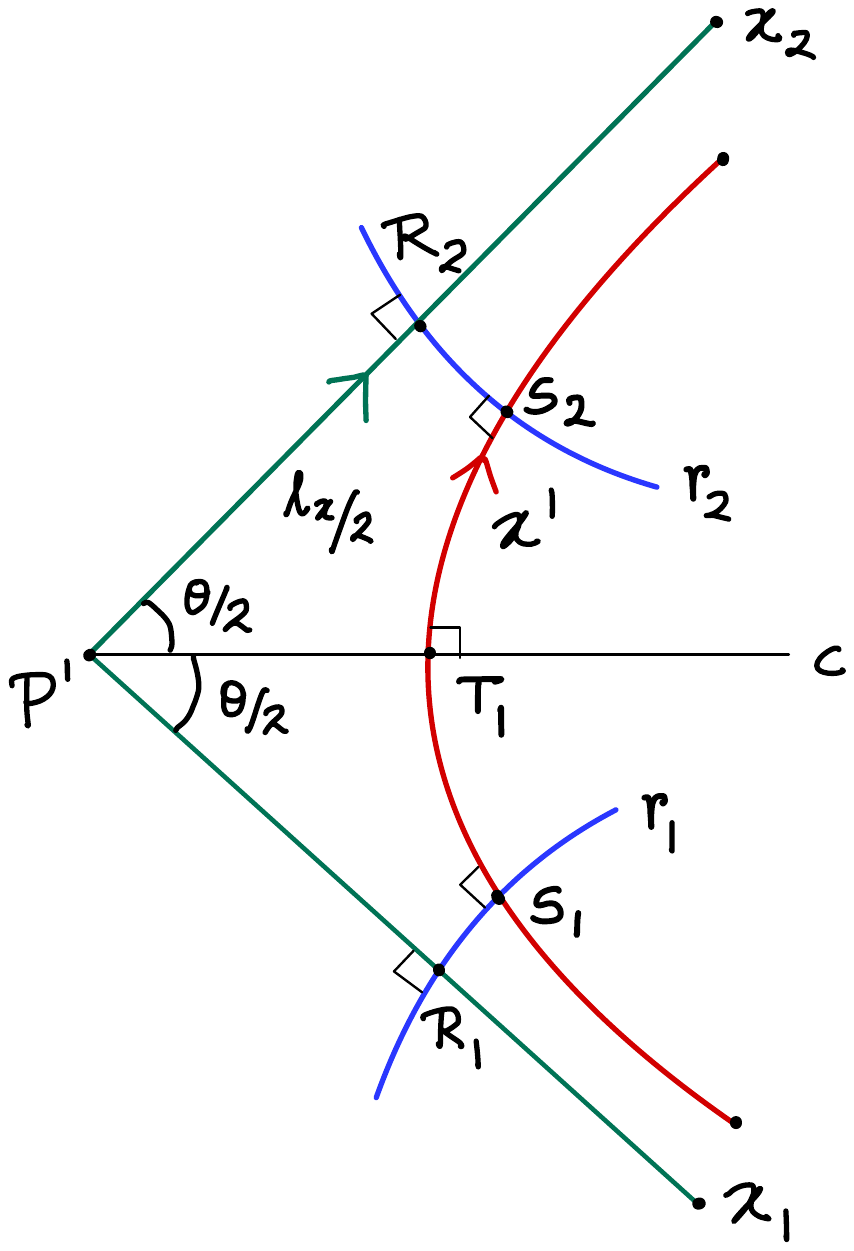}
\caption{The common perpendicular $x'$ of $r_1$, $r_2$, and $c$ is a full lift of $x^\sharp$.}
	\label{fig:teardropperp}
\end{figure}

\begin{lemma}\label{lem:geodesic}
The isometry $X$ is hyperbolic, and its axis $A_X$ is the geodesic $x'$, the common perpendicular of $r_1$, $r_2$, and $c.$ Moreover, $x'$ is a full lift of the unique geodesic freely homotopic to the teardrop geodesic $x$. The translation length of $X$ is equal to the distance between $r_1$ and $r_2$.
\end{lemma}
\begin{figure}[ht]
	\centering
	\includegraphics[width=0.4\textwidth]{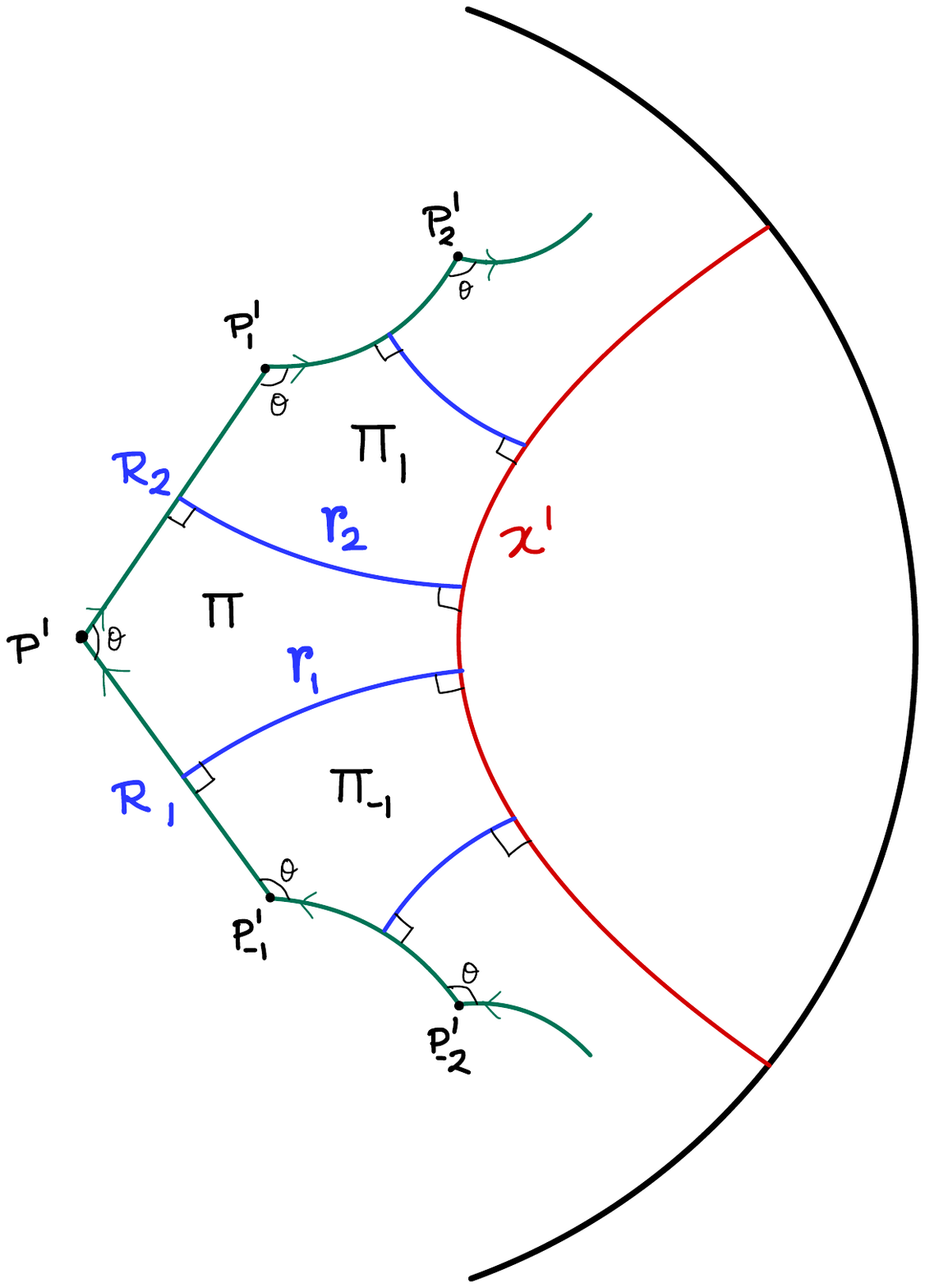}
\caption{Chain of pentagons.}
	\label{fig:teardroppent}
\end{figure}

\begin{proof}
Since $X=\rho_{r_1}\circ\rho_c$ is the composition of reflections in two distinct geodesics having the common perpendicular $x'$, it is an orientation-preserving hyperbolic isometry with axis $x'$.

Moreover, since $r_1=\rho_c(r_2)$, we have $\rho_{r_1}=\rho_c\circ\rho_{r_2}\circ\rho_c$, and therefore
\[
X=\rho_{r_1}\circ\rho_c=\rho_c\circ\rho_{r_2}=U^{-1},
\]
where $U=\rho_{r_2}\circ\rho_c$ is the covering transformation constructed in the proof of Lemma~\ref{lemma:common_perp}. In particular, $X$ is a covering transformation. Note also that $U$ maps $P'$ to the endpoint of the segment $x_2$, which is the endpoint of the lift of the loop $x$ starting at $P'$; hence $U$ is the deck transformation determined by the loop $x$ based at its vertex and the lift $P'$.

The translation length of $X$ is twice the distance between $r_1$ and $c$, which is also equal to the distance between $r_1$ and $r_2$. 

We claim that  $x'$ is at a bounded distance from a full lift of the teardrop geodesic $x$.  Indeed, consider the pentagon $\Pi$ with vertex $P'$, as shown in Figure~\ref{fig:teardroppent}. Observe that, except possibly for the angle at $P'$, all the angles in this pentagon are right angles.

Reflect the pentagon $\Pi$ across $r_1$ and across $r_2$ to obtain the pentagons $\Pi_{-1}$ and $\Pi_{1}$, respectively. By successively reflecting these pentagons across the appropriate lines (marked in blue in Figure~\ref{fig:teardroppent}), we construct a chain of pentagons, each with a side lying on the axis $A_X$. Also, observe that the other sides of the chain of pentagons (with appropriate direction) form a full lift of $x$.    This completes the proof of the claim.

Since $x'$ is a geodesic, the above claim  implies that $x'$ is a lift of the closed geodesic in the free homotopy class of $x$. 

Finally, note that the teardrop geodesic $x$ is the projection of the union of the segments $x_1$ and $x_2$ between the midpoints $R_1$ and $R_2$. Similarly, the geodesic in the free homotopy class of $x$ is the projection of the segment of $x'$ lying between the perpendiculars $r_1$ and $r_2$. 
\end{proof}

\begin{figure}[ht]
	\centering
	\includegraphics[width=0.4\textwidth]{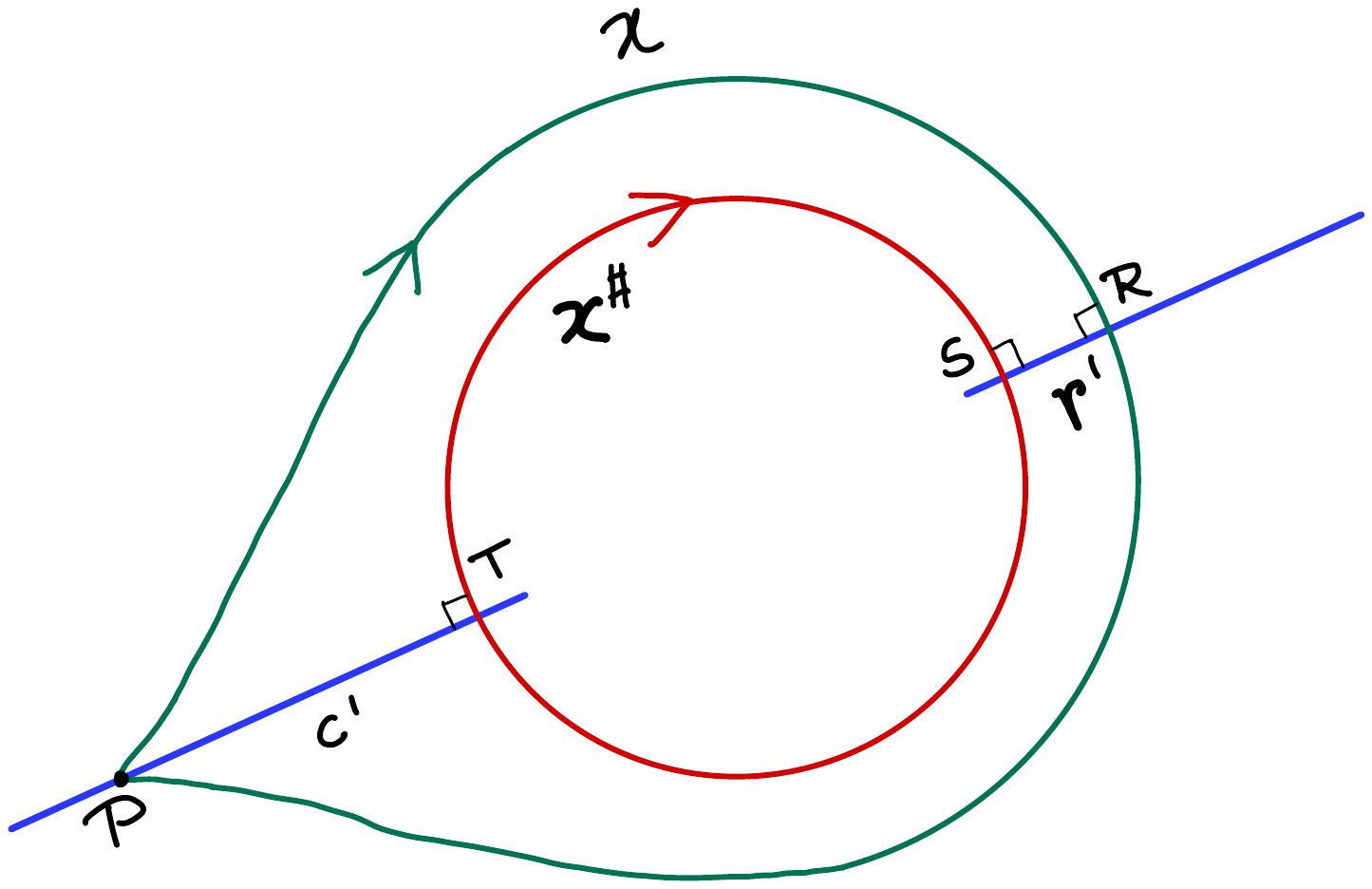}
\caption{The teardrop geodesic $x$ with vertex $P$, the geodesic $x^\sharp$ in its free homotopy class, and the perpendiculars $r'$ at $R$ and $c'$ at $P$ meeting $x^\sharp$ at $S$ and $T$.}
	\label{fig:teardroproj}
\end{figure}  

The following theorem describes a construction of the unique geodesic on the free homotopy class of a teardrop geodesic by considering the two lifts of the latter through the vertex and making a geometric construction of the universal cover that determines a lifting of the unique geodesic. 

Let $x$ be a teardrop geodesic on $\Sigma$ with vertex $P$. 
Denote by $R$ the point on $x$ at distance $\ell_x/2$ from $P$, measured along $x$. Denote by $r'$ the geodesic perpendicular to $x$ at $R$. Denote by $c'$ the geodesic that bisects $\theta_P$. See Figure~\ref{fig:teardroproj}.
Choose a lift $P'$ of $P$ to $\mathbb H$, and let $x_1$ and $x_2$ be the two directed lifts of $x$ through $P'$. For each $i$ in $\{1,2\}$, let $R_i$ be the midpoint of $x_i$, and let $r_i$ be the geodesic perpendicular to $x_i$ at $R_i$.  Let $c$ be the bisector of the angle from $x_1$  to $x_2$ at $P'$.

\begin{theorem}\label{thm:geodesic_representative_teardrop} With the notations above, we have.
\begin{enumerate}
\item  The common perpendicular $x'$ of $r_1$, $r_2$, is also perpendicular to  $c$ and projects to the unique closed geodesic $x^\sharp$ freely homotopic to $x$. Moreover, the length of $x^\sharp$ equals  the distance between $r_1$ and $r_2$.
\item The geodesic $r'$ meets $x^\sharp$ perpendicularly at a point $S$, and $c'$ meets $x^\sharp$ perpendicularly at a point $T$. The points $S$ and $T$ divide $x^\sharp$ into two arcs, each of length $\ell_{x^\sharp}/2$.(For proper powers, the arcs are understood along the parametrized geodesic, rather than as subsets of its image.)  See Figure~\ref{fig:teardroproj}.  
\end{enumerate}
\end{theorem}

\begin{proof}
   (1) follows directly from Lemma~\ref{lem:geodesic}, which  establishes that $x'$ is the axis of $X$, is perpendicular to  $r_1$, $r_2$, and $c$, and projects to $x^\sharp$, with $\ell_{x^\sharp}$ equal to the distance between $r_1$ and $r_2$.

For (2), for  $i\in\{1,2\}$, let $S_i$ denote the intersection of $r_i$ with $x'$,  and let $T_1$ denote the intersection of $c$ with $x'$. By Lemma~\ref{lem:geodesic}, $r_i$ meets $x'$  perpendicularly, so $S_i$ projects to a point $S$ on $x^\sharp$  where $r'$ meets $x^\sharp$ perpendicularly. Similarly, $T_1$  projects to a point $T$ on $x^\sharp$ where $c'$ meets $x^\sharp$ 
perpendicularly. Since $S_1$ and $T_1$ lie on $x'$ at distance $\ell_{x^\sharp}/2$ apart, their projections $S$ and $T$ divide $x^\sharp$ into two arcs, each of length $\ell_{x^\sharp}/2$.

\end{proof}

\begin{figure}[ht]
	\centering
	\includegraphics[width=0.4\textwidth]{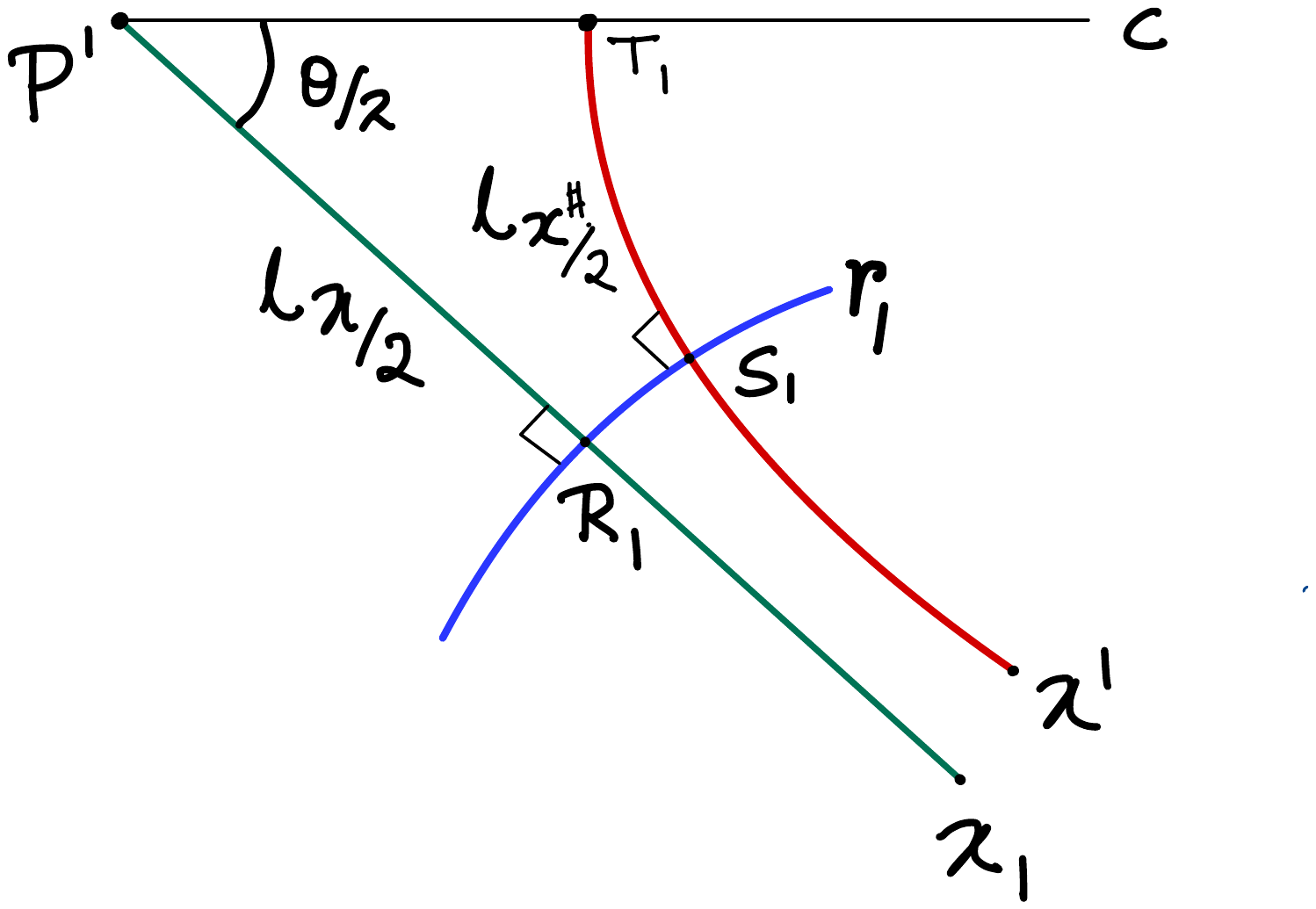}
\caption{Quadrilateral.}
	\label{fig:quadrilateral}
\end{figure}

Consider the quadrilateral of Figure~\ref{fig:quadrilateral}. By \cite[Theorem 7.17.1]{beardon2012geometry}, we have the following:

\begin{theorem}\label{thm:teardrop_metric}
Let $x$ be a teardrop geodesic on $\Sigma$ with vertex $P$. Then
\[
\sin\frac{\theta_P}{2}\cosh\frac{\ell_x}{2}
=
\cosh\frac{\ell_{x^\sharp}}{2}.
\]
In particular, $\sin\frac{\theta_P}{2}\cdot\cosh\frac{\ell_x}{2}
> 1$.
\end{theorem}

\begin{remark}
By Theorem~\ref{thm:teardrop_metric}, if $\theta_P$ is very
small, then the teardrop geodesic must be very long. In other words, the vertex angle measures the  ``loss'' of length from the teardrop to its geodesic representative.
\end{remark}

\begin{remark}
    The closed geodesic $x^\sharp$ is perpendicular to the geodesics $c'$ (the angle bisector at $P$) and $r'$ (the perpendicular through the ``midpoint'' $R$).  It would be interesting to determine sufficient conditions for $x^\sharp$ to be the shortest geodesic with that property. 
\end{remark}

   \begin{figure}[ht]
	\centering
	\includegraphics[trim={0 140 0 30},  width=.7\textwidth]{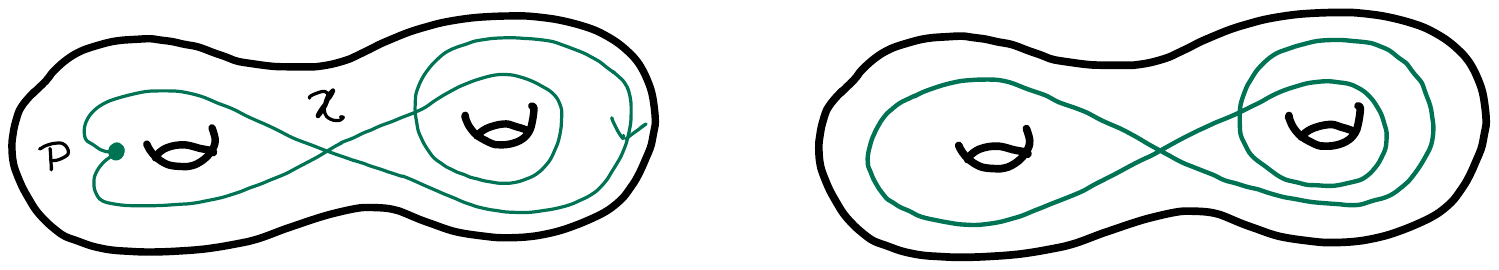}
\caption{A self-intersection teardrop geodesic and the unique geodesic in its free homotopy class.}
	\label{fig:teardropself}
\end{figure}

\section{Geometric definition of Turaev's Cobracket}\label{sec:definition}

In this section, we use the previous construction to define Turaev's cobracket in purely geometric terms.
Assume now that the orientable surface $\Sigma$, equipped with a complete metric of constant curvature $-1$ and geodesic boundary, has a preferred orientation.

Consider a closed directed geodesic $x$ on $\Sigma$, and a point $P$ on $x$ where $x$ passes more than once. We focus on the case where $x$ passes through $P$ exactly twice. (If $x$ passes through $P,$ $k > 2$ times, one must consider the $k$ small segments of $x$ at $P$ and apply the construction to each of the $\binom{k}{2}$ pairs.)

Write $x$ as the concatenation at $P$ of two directed piecewise geodesic loops, $a_P$ and $b_P$. By possibly relabeling, we may assume that the orientation of $\Sigma$ at $P$ agrees with that induced by the ordered pair (tangent vector of the $a_P$ branch at $P$, tangent vector of the $b_P$ branch at $P$). See Figure \ref{fig:splitting}.

\begin{figure}[ht]
	\centering
	\includegraphics[trim={0 80 0 50},  width=.7 \textwidth]{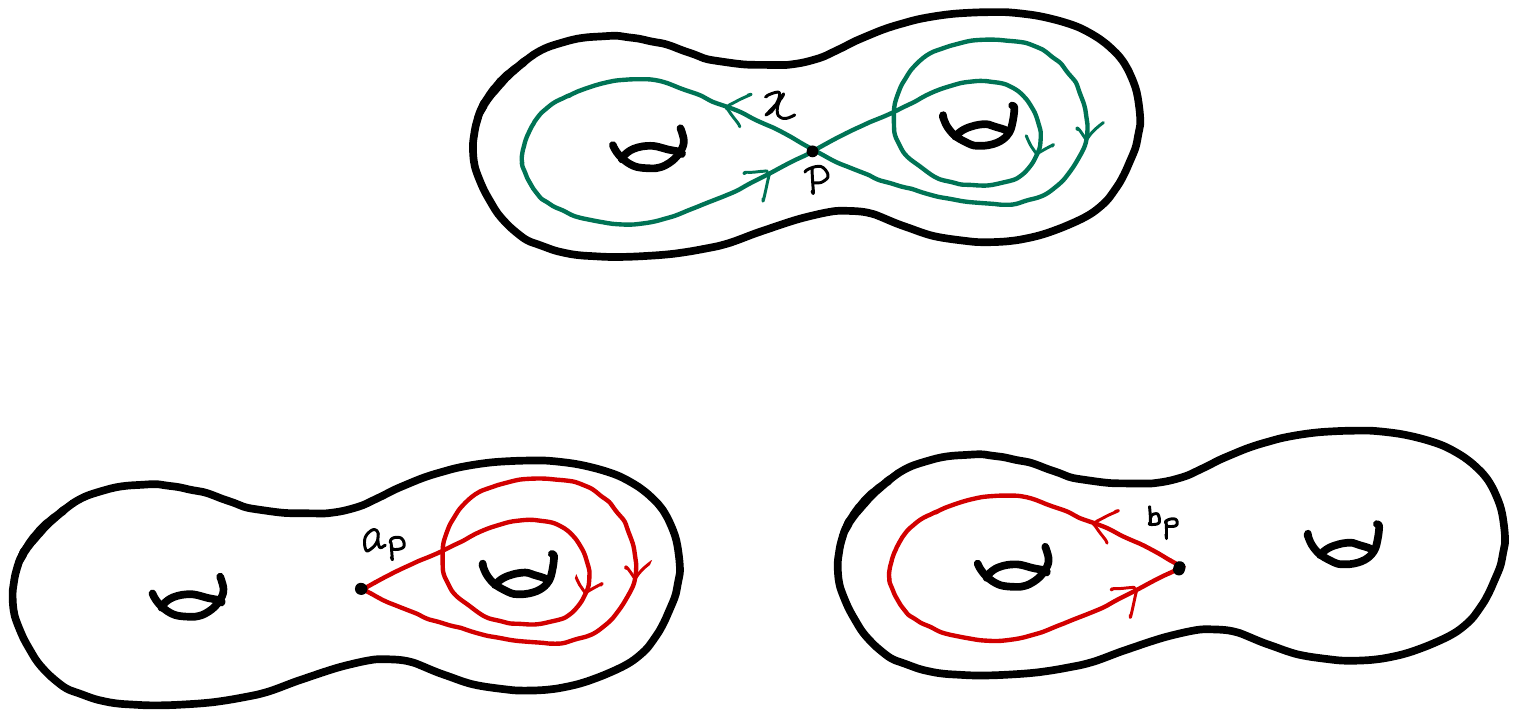}
\caption{Splitting of a geodesic into two teardrop geodesic at a self-intersection point.}
	\label{fig:splitting}
\end{figure}

Observe that $a_P$ and $b_P$ are teardrop geodesics with vertices at $P$.  The vertex angles of $a_P$ and $b_P$ at $P$ are vertical angles, hence both equal to $\theta_P$, and the two angles are bisected by the same complete geodesic; we denote by $c$ (a ray of) this common bisector.

Choose a point $P'$ in the universal cover $\mathbb{H}$, which is a lift of $P$. Consider the lifts of $a_P$ and $b_P$ passing through $P'$ as shown in Figure~\ref{fig:cobracketgeom}. Observe that Figure~\ref{fig:cobracketgeom} is the concatenation of two copies of Figure~\ref{fig:teardropperp} at $P'$.

\begin{figure}[ht]
	\centering
	\includegraphics[width=0.5\textwidth]{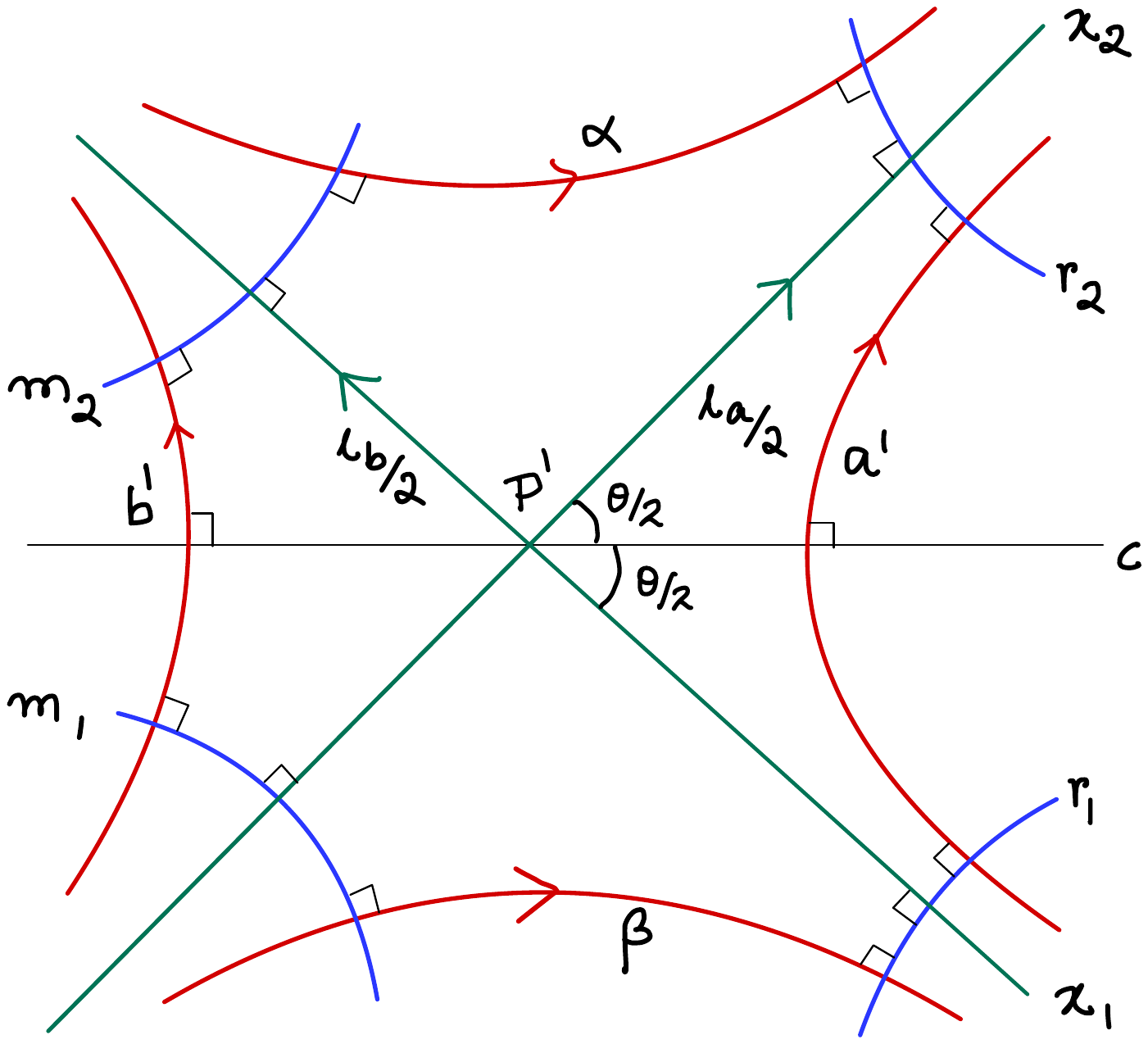}
\caption{Lift around a point $P.$ }
	\label{fig:cobracketgeom}
\end{figure}

We use the notations in Figure~\ref{fig:cobracketgeom}. Observe that the compositions $\rho_{r_2} \circ \rho_{m_1}$ and $\rho_{r_1} \circ \rho_{m_2}$ are covering transformations. Moreover, the segment $x_1$ is contained in the axis of $\rho_{r_1} \circ \rho_{m_2}$, and the segment $x_2$ is contained in the axis of $\rho_{r_2} \circ \rho_{m_1}$; see Figure~\ref{fig:cobracketgeom}. Consider the isometries $A$ and $B$ defined by
\[
A = \rho_{c} \circ \rho_{r_1} = \rho_{r_2} \circ \rho_{c},
\qquad
B = \rho_{c} \circ \rho_{m_1} = \rho_{m_2} \circ \rho_{c}.
\]
Applying the proof of Lemma~\ref{lemma:common_perp} to the teardrop geodesic $a_P$, we see that the isometry $A=\rho_c\circ\rho_{r_1}=\rho_{r_2}\circ\rho_c$ is precisely the covering transformation mapping the directed lift $x_1$ onto the directed lift $x_2$; similarly, $B$ is the covering transformation associated with the teardrop geodesic $b_P$. By Lemma~\ref{lem:geodesic}, their axes $a'$ and $b'$ project to the unique geodesics in the free homotopy classes of $a_P$ and $b_P$, respectively.

\begin{remark} Observe that the isometries 
\[ \rho_{r_2} \circ \rho_{m_2} = \rho_{r_2} \circ \rho_{c} \circ \rho_{c} \circ \rho_{m_2} = A B^{-1} \] 
and 
\[ \rho_{r_1} \circ \rho_{m_1} = \rho_{r_1} \circ \rho_{c} \circ \rho_{c} \circ \rho_{m_1} = A^{-1} B, \] 
are also covering transformations with axes $\alpha$ and $\beta$, respectively (see Figure~\ref{fig:cobracketgeom}). $\alpha$ and $a'$ share the perpendicular $r_2$, hence are disjoint.
\end{remark}

Denote by $V$ the vector space generated by $S$, the set of closed, directed, non-constant geodesics on $\Sigma$.   Consider an element $a \in S,$ and a transversal self-intersection point $P$ of $a$. If the multiplicity of $P$ is larger than $2,$ choose a pair of transversal intersecting branches. (For simplicity, we denote all these possibilities by $P \in \cap a$.)
For element of $\cap a$ , cutting $a$ at the two branches produces two directed teardrop geodesics. We denote them by $a_1^P$ and $a_2^P$, ordered so that the orientation determined by the branch of $a_1^P$ leaving $P$, followed by the branch of $a_2^P$ leaving $P$, agrees with the orientation of $\Sigma$.  Their geodesic representatives are then $(a_1^P)^\sharp$ and $(a_2^P)^\sharp$.

Define a map $\delta' \colon S \to V \otimes V$ by \[ \delta'(a) = \sum_{P \in \cap a} \left( \left(a_1^P\right)^\sharp \otimes \left(a_2^P\right)^\sharp - \left(a_2^P\right)^\sharp \otimes \left(a_1^P\right)^\sharp \right), \] where $\cap a$ denotes the set of all transverse pairs of branches of $a$. If  $a$ is regarded as a parametrized curve $a:S^1\to\Sigma$, an element of $\cap a$ is a pair of distinct points $s,t\in S^1$ such that $a(s)=a(t)$ with transverse tangent vectors $a'(s)$ and $a'(t)$ . Extend $\delta'$ to $V$ by linearity.

\begin{remark}\label{rem:power-cancellation}
If $a=b^k$, where $a$ and $b$ are closed geodesics, $b$ is not a proper power of another geodesic, and $k\geq 2$, then the minimal self-intersection number of $a$ is $k^2$ times the minimal self-intersection number of $b$ plus $(k-1)$, that is, $k^2 \mathrm{si}(b)+k-1$, where  $\mathrm{si}(b)$ denotes the minimal self-intersection number of $b$. 
By considering a small perturbation of $a$, we see that, after applying $f\otimes f$, the sum of the $(k-1)$ terms in the usual cobracket $\delta$ corresponding to the ``power'' points is
  \[
    \begin{aligned}
    &\left(b\right)^\sharp \otimes \left(b^{k-1}\right)^\sharp - \left(b^{k-1}\right)^\sharp \otimes \left(b\right)^\sharp \\
    +\ &\left(b^2\right)^\sharp \otimes \left(b^{k-2}\right)^\sharp - \left(b^{k-2}\right)^\sharp \otimes \left(b^2\right)^\sharp \\
    +\ &\dots \\
    +\ &\left(b^{k-1}\right)^\sharp \otimes \left(b\right)^\sharp - \left(b\right)^\sharp \otimes \left(b^{k-1}\right)^\sharp
    \ =\ 0.
    \end{aligned}
    \]
    Then, the $(k-1)$ intersection points coming from the curve wrapping around itself can be disregarded when computing the cobracket. 
\end{remark}

We follow the notations of  Appendix~\ref{app:Lie}. The unique-geodesic-representative  property of hyperbolic surfaces induces an isomorphism 
$f\colon W'\to V$.

\begin{theorem}\label{thm:isomorphism}
The isomorphism $f\colon W'\to V$ is an isomorphism of Lie  coalgebras between Turaev's coalgebra $(W',\delta)$ defined  in Appendix~\ref{app:Lie} and $(V,\delta')$. In symbols,
\[
\delta'\circ f=(f\otimes f)\circ\delta.
\]
\end{theorem}

\begin{proof}
It is enough to verify the identity on the elements of the natural basis of $W'$. Let $x$ be a closed directed geodesic. For each $P\in\cap x$, let $x_1^P$ and $x_2^P$ be the two directed teardrop geodesics obtained by cutting $x$ at $P$, ordered as in the
definitions of $\delta$ and $\delta'$. By definition of $f$,
\[
f\left(\langle x_1^P\rangle_0\right)
=
\left(x_1^P\right)^\sharp
\qquad\text{and}\qquad
f\left(\langle x_2^P\rangle_0\right)
=
\left(x_2^P\right)^\sharp.
\]
If $x$ is a proper power, Remark~\ref{rem:power-cancellation} shows that the additional intersection points introduced by a generic perturbation have total contribution zero. Hence, in all cases, the cobracket may be computed using only the transverse pairs of branches.

Therefore,
\[
\begin{aligned}
(f\otimes f)\bigl(\delta(\langle x\rangle)\bigr)
&=
\sum_{P\in\cap x}
\left(
\left(x_1^P\right)^\sharp\otimes
\left(x_2^P\right)^\sharp
-
\left(x_2^P\right)^\sharp\otimes
\left(x_1^P\right)^\sharp
\right)\\
&=
\delta'(x)
=
\delta'\bigl(f(\langle x\rangle)\bigr).
\end{aligned}
\]
The ordering conventions agree by construction. Hence
\[
\delta'\circ f=(f\otimes f)\circ\delta.
\]
\end{proof}

\begin{remark}
Let $P$ be a self-intersection point of the closed geodesic $x$, and let $x_1^P$ and $x_2^P$ be the two teardrop geodesics determined by $P$. Then

\[
\sin\left(\frac{\theta_P}{2}\right)
=
\frac{
\cosh\left(\frac{\ell_{(x_1^P)^\sharp}}{2}\right)
}{
\cosh\left(\frac{\ell_{x_1^P}}{2}\right)
}
=
\frac{
\cosh\left(\frac{\ell_{(x_2^P)^\sharp}}{2}\right)
}{
\cosh\left(\frac{\ell_{x_2^P}}{2}\right)
}.
\]

Thus the two teardrops determined by $P$ have the same normalized geodesic-to-teardrop ratio: the ratio of the hyperbolic cosine of the length of the geodesic representative to the hyperbolic cosine of the length of the teardrop is exactly $\sin(\theta_P/2)$.
\end{remark}

\section{Application: Cancellation pairs and rigidity}

We now derive a metric rigidity consequence of the cobracket construction. Let $x$ be a closed directed geodesic on $\Sigma$, and let $P$ be a self-intersection point of $x$. Cutting $x$ at $P$ produces the two directed teardrop geodesics $x_1^P$ and $x_2^P$, ordered as above. Their associated angles are both $\theta_P$. Their lengths satisfy
\[
\ell_{x_1^P}+\ell_{x_2^P}=\ell_x.
\]

For a metric $Y$ in the Teichmüller space of $\Sigma$, we use $\ell_{x_1^P}(Y)$, $\ell_{x_2^P}(Y)$, and $\theta_P(Y)$ for the corresponding lengths and angle determined by the $Y$-geodesic representative of $x$.

Let $x$ be a non-simple, non proper power, free homotopy class on $\Sigma$. 

For every complete hyperbolic metric $Y$ on $\Sigma$  the $Y$-geodesic representative of $x$ realizes the minimal self-intersection number of $x$ \cite{hass1985intersections}, and its transverse pairs of branches can be canonically identified as $Y$ varies.  In the next proposition,  $P$ and $Q$ are two of such transverse pairs of branches, and by $x_1^P$, $x_2^P$, $x_1^Q$, $x_2^Q$ denote the associated teardrop geodesics, ordered as in Section 3. The free homotopy classes of these teardrop loops are determined by the splitting of $x$ at $P$ and $Q$ and do not depend on the metric; only their lengths and angles do. We write $\ell_{x_1^P}(Y)$, $\theta_P(Y)$, and so on, for the corresponding lengths and angles.

\begin{proposition}\label{prop:cancellation_rigidity}
Let $x$ be a  free homotopy class on $\Sigma$, and let $P$ and $Q$ be two transverse pairs of branches. Suppose that the terms at $P$ and $Q$ cancel in $\delta'(x)$. Then, for every complete hyperbolic metric $Y$ on $\Sigma$ with no cusps and with possibly empty geodesic boundary,
\begin{enumerate}
    \item $\ell_{x_1^P}(Y)=\ell_{x_2^Q}(Y);$
    \item $ \theta_P(Y)=\theta_Q(Y).$ 
\end{enumerate}
\end{proposition}

\begin{proof}
Fix the metric $Y$ and suppress the dependence on $Y$ throughout.
The cancellation hypothesis gives
\begin{align}
\ell_{\left(x_1^P\right)^\sharp}
&=
\ell_{\left(x_2^Q\right)^\sharp},
\label{eq:cancel_1P_2Q}\\
\ell_{\left(x_2^P\right)^\sharp}
&=
\ell_{\left(x_1^Q\right)^\sharp}.
\label{eq:cancel_2P_1Q}
\end{align}
Applying Theorem~\ref{thm:teardrop_metric} to the four teardrop geodesics gives
\begin{align}
\sin\frac{\theta_P}{2}
\cosh\frac{\ell_{x_1^P}}{2}
&=
\cosh\frac{\ell_{\left(x_1^P\right)^\sharp}}{2},
\label{eq:T_1P}\\
\sin\frac{\theta_P}{2}
\cosh\frac{\ell_{x_2^P}}{2}
&=
\cosh\frac{\ell_{\left(x_2^P\right)^\sharp}}{2},
\label{eq:T_2P}\\
\sin\frac{\theta_Q}{2}
\cosh\frac{\ell_{x_1^Q}}{2}
&=
\cosh\frac{\ell_{\left(x_1^Q\right)^\sharp}}{2},
\label{eq:T_1Q}\\
\sin\frac{\theta_Q}{2}
\cosh\frac{\ell_{x_2^Q}}{2}
&=
\cosh\frac{\ell_{\left(x_2^Q\right)^\sharp}}{2}.
\label{eq:T_2Q}
\end{align}
Combining \eqref{eq:cancel_1P_2Q}, \eqref{eq:T_1P}, and \eqref{eq:T_2Q}, and using
$\ell_{x_2^Q}=\ell_x-\ell_{x_1^Q},$ we obtain
\begin{equation}\label{eq:combined_A}
\sin\frac{\theta_P}{2}
\cosh\frac{\ell_{x_1^P}}{2}
=
\sin\frac{\theta_Q}{2}
\cosh\frac{\ell_x-\ell_{x_1^Q}}{2}.
\end{equation}
Similarly, combining \eqref{eq:cancel_2P_1Q}, \eqref{eq:T_2P}, and \eqref{eq:T_1Q}, and using $\ell_{x_2^P}=\ell_x-\ell_{x_1^P},$ gives
\begin{equation}\label{eq:combined_B}
\sin\frac{\theta_P}{2}
\cosh\frac{\ell_x-\ell_{x_1^P}}{2}
=
\sin\frac{\theta_Q}{2}
\cosh\frac{\ell_{x_1^Q}}{2}.
\end{equation}

Dividing \eqref{eq:combined_A} by \eqref{eq:combined_B} gives
\[
\frac{\cosh(\ell_{x_1^P}/2)}
{\cosh((\ell_x-\ell_{x_1^P})/2)}
=
\frac{\cosh((\ell_x-\ell_{x_1^Q})/2)}
{\cosh(\ell_{x_1^Q}/2)}.
\]
Define
\[
h(s)=
\frac{\cosh(s/2)}
{\cosh((\ell_x-s)/2)},
\qquad 0<s<\ell_x.
\]
A direct computation gives $h'(s)>0.$ Thus $h$ is strictly increasing and therefore injective. The preceding
equality is
\[
h\left(\ell_{x_1^P}\right)
=
h\left(\ell_x-\ell_{x_1^Q}\right),
\]
and hence
\[
\ell_{x_1^P}
=
\ell_x-\ell_{x_1^Q}.
\]
This proves
\[
\ell_{x_1^P}+\ell_{x_1^Q}=\ell_x.
\]

Substituting this equality into \eqref{eq:combined_A} gives
\[
\sin\frac{\theta_P}{2}
=
\sin\frac{\theta_Q}{2}.
\]
Since $\theta_P,\theta_Q\in(0,\pi)$, it follows that
\[
\theta_P=\theta_Q.
\]
Since $Y$ was arbitrary, the result holds for any $Y$  in the Teichmüller space of $\Sigma$
\end{proof}

\appendix

\section{Turaev cobracket}\label{app:Lie}

A \textit{Lie coalgebra} is a vector space $V$ equipped with a linear map $\delta: V \to V \otimes V$ satisfying:
\begin{itemize}
  \item (Co-skew symmetry) $\sigma \circ \delta = -\delta$.
  \item (Co-Jacobi identity) $(\tau^2 + \tau + 1) \circ (\mathrm{Id} \otimes \delta) \circ \delta = 0$.
\end{itemize}
Here, $\sigma(x \otimes y) = y \otimes x$ and $\tau(x \otimes y \otimes z) = z \otimes x \otimes y$.

Let $W$ denote the vector space with basis the set of free homotopy classes of closed directed curves on $\Sigma$. In \cite{turaev1991skein}, Turaev defined a Lie coalgebra on $W'$, the quotient of $W$ by the (space generated by the) class of the contractible loop, as follows. Let $\alpha$ be a representative of a free homotopy class whose self-intersections are transverse double points. Each self-intersection point $P$ determines two loops based at $P$, denoted by $\alpha_P^1$ and $\alpha_P^2$, such that $\alpha$ is the loop product $\alpha_P^1\alpha_P^2$ at $P$. We order these loops so that the ordered pair consisting of the branch of $\alpha_P^1$ leaving $P$ and the branch of $\alpha_P^2$ leaving $P$ agrees with the orientation of $\Sigma$.

Their free homotopy classes, with the contractible class declared to be $0$, are denoted by $\langle \alpha_P^1\rangle_0$ and $\langle \alpha_P^2\rangle_0$. Define
\[
\delta(\langle \alpha\rangle)
=
\sum_{P\in \cap \alpha}
\left(
\langle \alpha_P^1\rangle_0 \otimes \langle \alpha_P^2\rangle_0
-
\langle \alpha_P^2\rangle_0 \otimes \langle \alpha_P^1\rangle_0
\right),
\]
where $\cap \alpha$ denotes the set of self-intersection points of the chosen representative $\alpha$.

The cobracket $\delta$ is well defined. Extend $\delta$ to $W'$ by linearity. This map is  co-skew symmetric and satisfies the co-Jacobi identity. 

\section*{Acknowledgement} The authors used AI tools to assist with LaTeX typesetting, grammar and language editing, and checking and correcting mathematical arguments. The authors independently verified the final manuscript and take full responsibility for its contents.
 \bibliographystyle{plain}

\bibliography{gla.bib}

@article{kabiraj2018equal,
  title={Equal angles of intersecting geodesics for every hyperbolic metric.},
  author={Kabiraj, Arpan},
  journal={New York Journal of Mathematics},
  volume={24},
  pages={167–181},
  year={2018}
}

@article{chas1999string,
  title={String topology},
  author={Chas, Moira and Sullivan, Dennis},
  journal={arXiv preprint math/9911159},
  year={1999}
}

@article{chas2023explicitliebracketclosed,
      title={Explicit Lie bracket of closed geodesics on a hyperbolic surface with applications}, 
      author={Moira Chas and Arpan Kabiraj},
      journal={arXiv preprint math/2308.02006, to appear in Annales de l'Institut Fourier.},
      year={2023},
      eprint={2308.02006},
      archivePrefix={arXiv},
      primaryClass={math.GT},
      url={https://arxiv.org/abs/2308.02006}, 
      
}

@article{chas2004combinatorial,
  title={Combinatorial {Lie} bialgebras of curves on surfaces},
  author={Chas, Moira},
  journal={Topology},
  volume={43},
  number={3},
  pages={543--568},
  year={2004},
  publisher={Elsevier}
}

@article{chas2010minimal,
	Author = {Chas, Moira},
	Journal = {Geometriae Dedicata},
	Number = {1},
	Pages = {25--60},
	Publisher = {Springer},
	Title = {Minimal intersection of curves on surfaces},
	Volume = {144},
	Year = {2010}}

@article{chas2010algebraic,
	Author = {Chas, Moira and Krongold, Fabiana},
	Journal = {Journal of Topology and Analysis},
	Number = {03},
	Pages = {395--417},
	Publisher = {World Scientific},
	Title = {An algebraic characterization of simple closed curves on surfaces with boundary},
	Volume = {2},
	Year = {2010}}

@article{chas2015algebraic,
	Author = {Chas, Moira and Krongold, Fabiana},
	Journal = {Journal of Topology},
	Number = {1},
	Pages = {91--104},
	Publisher = {Oxford University Press},
	Title = {Algebraic characterization of simple closed curves via Turaev's cobracket},
	Volume = {9},
	Year = {2015}}

@article{chas2016extended,
  title={The extended Goldman bracket determines intersection numbers for surfaces and orbifolds},
  author={Chas, Moira and Gadgil, Siddhartha},
  journal={Algebraic \& Geometric Topology},
  volume={16},
  number={5},
  pages={2813--2838},
  year={2016},
  publisher={Mathematical Sciences Publishers}
}

@article{ledonne2008lie,
  title={On {Lie} bialgebras of loops on orientable surfaces},
  author={Le Donne, Attilio},
  journal={Journal of Knot Theory and Its Ramifications},
  volume={17},
  number={03},
  pages={351--359},
  year={2008},
  publisher={World Scientific}
}

@book{beardon2012geometry,
	Author = {Beardon, Alan F},
	Publisher = {Springer Science \& Business Media},
	Title = {The geometry of discrete groups},
	Volume = {91},
	Year = {2012}}

@article{hass1985intersections,
  author  = {Hass, Joel and Scott, Peter},
  title   = {Intersections of curves on surfaces},
  journal = {Israel Journal of Mathematics},
  volume  = {51},
  pages   = {90--120},
  year    = {1985}
}

@article{goldman_invariant_1986,
	Author = {Goldman, William},
	Journal = {Inventiones Mathematicae},
	Number = {2},
	Pages = {263--302},
	Title = {Invariant functions on {Lie} groups and {Hamiltonian} flows of surface group representations},
	Volume = {85},
	Year = {1986}}

@article{turaev1991skein,
  title   = {Skein quantization of Poisson algebras of loops on surfaces},
  author  = {Turaev, Vladimir G.},
  journal = {Annales scientifiques de l'{\'E}cole normale sup{\'e}rieure},
  volume  = {24},
  number  = {6},
  pages   = {635--704},
  year    = {1991}
}

@article{cahn2013generalization,
  title={A generalization of the Turaev cobracket and the minimal self-intersection number of a curve on a surface},
  author={Cahn, Patricia},
  journal={New York J. Math},
  volume={19},
  pages={253--283},
  year={2013}
}

\end{document}